\documentclass[reqno,11pt]{amsart}
\usepackage[hmargin=2.6cm,bmargin=2.55cm,tmargin=3.05cm]{geometry}
\usepackage{xcolor}
\usepackage{hyperref}
\usepackage{amssymb}
\usepackage{graphicx}
\usepackage{enumerate}
\usepackage{mathtools}
\mathtoolsset{showonlyrefs}
\begin{document}

\title[]
{On the existence of Nehari ground states for the Nonlinear Schr\"{o}dinger Equation on Discrete Graphs}
\author[]
{Setenay Akduman,  Matthias Hofmann, Sedef Karakili\c{c}}

\address{Setenay Akduman, Department of Mathematics, Izmir Democracy University, Izmir, 35140, Turkey }

\address{Sedef Karakili\c{c}, 
	Dokuz Eyl\"{u}l University, Faculty of Science, Department of Mathematics, Izmir Turkey}

\address{Matthias Hofmann, Fakult\"at Mathematik und Informatik, Fern\-Universit\"at in Hagen, D-58084 Hagen, Germany}

\dedicatory{In memory of Alexander Pankov}

\thanks{\emph{Acknowledgements.}  S.~Akduman and S. Karakili\c{c} were supported by the COST Action~24122. M. Hofmann was supported by the Portuguese government through FCT - Fundação para a Ciência e a Tecnologia, I.P., under the project SpectralOPs with reference 2023.13921.PEX. S. Akduman is grateful to Ognjen Milatovic and Stephen Shipman for interesting discussions.}



\newcommand{\diam}{\operatorname{diam}}

\newcommand{\MH}[1]{{\color{blue} MH: #1}}
\newcommand{\SA}[1]{{\color{red} SA: #1}}
\newcommand{\SK}[1]{{\color{purple} SK: #1}}

\numberwithin{equation}{section}
\newtheorem{theorem}{Theorem}[section]
\newtheorem{corollary}[theorem]{Corollary}
\newtheorem{proposition}[theorem]{Proposition}
\newtheorem{definition}[theorem]{Definition}
\newtheorem{assumption}[theorem]{Assumption}
\newtheorem{lemma}[theorem]{Lemma}

\theoremstyle{definition}
\newtheorem{example}[theorem]{Example}
\newtheorem{remark}[theorem]{Remark}

\allowdisplaybreaks

\begin{abstract}
We study standing waves for the nonlinear Schr\"{o}dinger equation on a discrete graph. We characterize for a self-adjoint realizations of Schrödinger operators conditions related with the geometry of the graph that guarantee discreteness of the spectrum and study ground states on the generalized Nehari manifold in order to prove the existence of standing wave solutions in the self-focusing and defocusing cases. In this context, we show properties of the solutions, such as integrability. Finally, we discuss decay properties of solutions and the bifurcation of solutions from the trivial solution.
\end{abstract}
\maketitle

\section{Introduction}\label{s1}

The study of nonlinear partial differential equations on discrete structures has garnered significant attention in recent years, driven by both theoretical interest and practical applications in fields such as quantum mechanics, network theory, and nonlinear optics. Among these, the nonlinear Schrödinger equation (NLS) stands out as a fundamental model describing wave propagation in nonlinear media:
\begin{equation*}
(-\Delta +V) \psi +  f(\cdot, \psi)=0,
\end{equation*}
where $\psi=\psi(t, \cdot)$ is a complex-valued wave function, $\Delta$ denotes the Laplacian (or a discrete Laplacian on graphs, which we will introduce later), $V$ a real-valued potential, and $f$ characterizes the nonlinearity, often taken as a power function of the form $f(\cdot, \psi) = g(\cdot) |\psi|^{p-2} u$ with $p >1$. When the NLS equation is considered on discrete graphs, it models wave dynamics in structured media such as optical lattices and photonic crystals.

Prominent physical realizations of this model are found in context of Bose-Einstein condensates and Kerr waveguides, where the refractive index of the medium depends on the intensity of the light. In such settings, the NLS equation governs the evolution of optical pulses, capturing phenomena such as self-focusing, soliton formation, and modulational instability.  We refer to \cite{led-08} for further information on the model.

By employing variational techniques and critical point theory, we establish conditions under which ground state solutions for the NLS equation exist on discrete graphs, highlighting the influence of the graph topology and the nonlinearity of the equation. When considered on discrete graphs, the nonlinear Schrödinger equation (NLS) reveals a rich interplay between the geometry of the underlying graph and the analytical properties of its solutions. The discrete structure introduces new challenges and phenomena not present in the continuous setting, such as localization effects, spectral gaps, and topological constraints. For a comprehensive introduction to the NLS equation in both discrete and continuous frameworks, we refer to \cite{ab-pr-tr}.

A central concept in the analysis of nonlinear equations is the notion of ground states—solutions that minimize the associated energy functional, which can be used to find critical points. Among the variational techniques developed to identify such solutions, the Nehari manifold method has emerged as a particularly effective tool. Originally introduced by Zeev Nehari in \cite{Nehari1960} in the context of second-order ordinary differential equations, the method was later extended to a broader class of nonlinear elliptic problems. The classical Nehari manifold is defined as a natural constraint set where the energy functional's derivative vanishes in the direction of the function itself, allowing for the identification of nontrivial critical points. We refer for an introduction to the Nehari methods and their developments in \cite{sz-weth-10}.

To address more complex variational problems, particularly those involving strongly indefinite functionals, the generalized Nehari manifold was introduced in \cite{pankovphotonic} for the periodic discrete NLS equation. This extension, sometimes referred to as the Nehari–Pankov manifold, adapts the classical framework to settings where the energy functional is not bounded below on the entire space, and the origin is a saddle point rather than a local minimum. The generalized manifold is constructed by decomposing the underlying Hilbert space into orthogonal subspaces and imposing orthogonality conditions on the derivative of the functional. This approach was further developed in \cite{dp-kr-sz} applied to semilinear Schrödinger equations with weak monotonicity conditions.

The generalized Nehari manifold has since become a powerful tool in the study of nonlinear PDEs, enabling the identification of ground states and multiple solutions in settings where classical methods are insufficient. Its flexibility and effectiveness have made it particularly valuable in problems involving complex geometries.

This is especially evident when the NLS equation is considered on discrete graphs or quantum network and new analytical challenges appear, particularly due to the presence of nonlinearities. The existence of solutions in such models and stability of the ground states were previously studied in \cite{pel-05}. In \cite{ke-haeseler}, the connection between the existence of solutions with specific properties and the spectrum on infinite graphs is investigated. In \cite{pankovperiodic2}, the existence of nontrivial exponentially decaying solutions to periodic stationary discrete NLS equations was given. In a setting where the potential is unbounded, some elementary existence results for standing wave solutions of discrete NLS equations were shown in \cite{pan-10, pan-13, pankovphotonic}.

 In the following, we will investigate in a similar spirit the existence of solutions to the NLS equation on combinatorial graphs, and characterize several geometric assumptions that allow us to recover questions regarding the existence of solitons. We plan to address these questions for graphs with finite and infinite measures. In this context, we extend the spectral theory from \cite{ke-le 2}, developed for graphs with finite measure, to graphs with infinite measure, with the key results being
\begin{itemize}
\item continuous and compact imbeddings of the energy space to weighted $\ell^p_m$-spaces for $p\in [1,\infty]$;
\item development of conditions on the potential in order to guarantee the discreteness of the spectrum. 
\end{itemize}
The obtained results contribute to the development of a unified and extensive framework for graphs with both finite and infinite measures, enabling us to  apply the critical point theory (see \cite{sz-weth-10}) to the NLS energy functional on graphs via the generalized Nehari manifold approach 
to investigate
\begin{itemize}
\item the existence of solutions to the NLS equation,
\item integrability properties of solutions,
\item relations to decay properties of eigenfunctions of Schrödinger operators,
\item bifurcation of solutions from the trivial solution.
\end{itemize}

Similar results were achieved previously in the metric graph case \cite{akd-pank} based on spectral theoretical results established in \cite{pankovakduman} for infinitely growing potential. In the related paper, the authors showed that the growth assumption on the potential is, in fact, a necessary and sufficient condition for the discreteness of the spectrum.  In the domain case, this was shown in \cite{Molchanov}. A novelty in our considered setting is the inclusion of geometric conditions in the framework to obtain results, and we conjecture that similar phenomena could be observed in the continuous case as well.

The article here adapts the method from \cite{akd-pank} in the discrete graph setting. Since the original method was strongly dependent on results from the domain case, let us summarize a few notable differences: 
\begin{itemize}
\item a spectral theoretical framework needed to be developed in order to include a large class of graphs in one unified approach;
\item the integrability of solutions in the metric graph case is related to the exponential growth of a graph. In the discrete case, this phenomenon is replaced by a volume growth assumption; 
\item results from the domain case can not as easily be adapted in this case, and the assumptions accordingly changed to account for this.
\end{itemize}




Our article is structured as follows.
In Section~\ref{sec:mainresults}, we summarize the framework and the main results including the existence and bifurcation results. In Section~\ref{spectraltheorysection}, we develop the spectral-theoretic results for the proofs of the main results.  The remaining sections are dedicated to the proofs of the main results.

\section{Formulation of problem and main results}\label{sec:mainresults}
\subsection{Setting the stage}
Let $\mathcal{V}$ be an infinite countable set and $m:\mathcal{V}\to (0,\infty)$ define a  measure on $\mathcal{V}$ via $$m(A):=\sum_{x\in A} m(x),$$ for any subset $A$ of $\mathcal{V}$. We say that a set $A$ of $\mathcal V$ has finite measure if $m(A)< \infty$ and say that $m$ is a finite measure if $m(\mathcal{V})<\infty$. We assume that the graph is a weighted graph   $\Gamma$ over the  measure space $(\mathcal{V},m)$  (see e.g. \cite{ke-le 2}, \cite{ke-le}  and \cite{KeLeWo}). More precisely; $\Gamma$ is determined by a pair $(b,c)$ consisting of two  maps; the edge weight $b:\mathcal{V} \times \mathcal{V}\to [0,\infty)$ and the killing term $c:\mathcal{V}\to [0,\infty)$  satisfying the following  properties: 

\bigskip
\begin{assumption}[Assumptions on the edge weights.] \label{ass:edgeweights} \ \ 
\begin{itemize}\em

	\item [$(b_0)$] (vanishing on the diagonal) $b(x,x)=0$ for all $x\in \mathcal{V}$
	\item[$(b_1)$] (symmetry) $b(x,y)=b(y,x)$ for all $x,y\in \mathcal{V}$,
	\item[$(b_2)$] (summability) for all $x\in \mathcal{V}$, we have  $\sum\limits_{y\sim x} b(x,y) < \infty$, where $y\sim x$ if and only if $x$ and $y$ are adjacent; that is, $b(x,y)>0$,
\end{itemize}
\end{assumption}
These conditions allow us to define an essentially self-adjoint operator in terms of quadratic forms as in \cite{ke-le-12}.


Let $\ell^p_m$ stand for the Banach space of all $u:\mathcal{V}\to \mathbb{R}$ such that 
\begin{equation*}
\|u\|^p_{\ell^p_m}=\sum\limits_{x\in \mathcal{V}}m(x)|u(x)|^p<\infty,
\end{equation*}
for $p\in[1,\infty)$ and  $\ell^2_m$ be a real Hilbert space with the inner product
\begin{equation*}
(u,v)_m=\sum\limits_{x\in \mathcal{V}}m(x)u(x)v(x).
\end{equation*}
If $m\equiv 1$, we drop the index $m$ in the notation of these spaces. $\ell^{\infty}$ denotes the space of bounded functions on $\mathcal{V}$ endowed with the \normalfont sup-norm
$$\|u\|_{\infty}:=\sup\limits_{x\in\mathcal{V}}|u(x)|.$$
For any function $u:\mathcal{V}\to \mathbb{R}$, the $m$-Laplacian of $u$ is defined as
$$\Delta u(x)=\dfrac{1}{m(x)}\sum\limits_{y\sim x}b(x,y)(u(y)-u(x)).$$
We use the notation $C_c(\mathcal{V})$ for  the vector space of finitely supported functions on $\mathcal{V}$. 
First we introduce the operator $L_0$ in $\ell^2_m$ with domain $D(L_0)=C_c(\mathcal{V})$  defined by

\begin{equation}\label{schrodinger}
L_0u(x)=-\Delta u(x) +V(x)u(x),
\end{equation}
for each vertex $x\in \mathcal{V}$, where $V(x):=\frac{c(x)+m(x)}{m(x)}:\mathcal{V}\to [1,\infty)$. Applying the discrete analogue of the Green`s formula (see  Proposition 3.3 in \cite{ke-le-12}), one can show that $L_0$ is a symmetric operator. Moreover, one easily verifies that $L_0 \ge I$, where $I$ is the identity operator on $\ell^2_m$.\\

For a general reference on Dirichlet forms on combinatorial graphs we refer to \cite{KeLeWo}. Let $q_0$ be the quadratic form associated to the operator $L_0$ defined by  
 $q_0(u)=(L_0u,u),u\in D(L_0)$, $q_0$ is closable and 
 $$q_0(u)\geq \| u\|^2_{\ell^2_m},$$ for each $u\in D(L_0)$. Let us denote the closure of $q_0$ by $q^{(D)}$ with the domain $D(q^{(D)})$, which is called the \emph{Dirichlet form}, and the \emph{Neumann form}
\begin{eqnarray*}
q^{(N)}(u)&=&\dfrac{1}{2}\sum\limits_{x,y\in\mathcal{V},\, y \sim x}b(x,y)(u(y)-u(x))^2+\sum\limits_{x\in\mathcal{V}}m(x)V(x)u^2(x),\\
D(q^{(N)})&=&\left \{u\in \ell^2_m(\mathcal V) : \dfrac{1}{2}\sum\limits_{x,y\in\mathcal{V},\, y \sim x}b(x,y)(u(y)-u(x))^2+\sum\limits_{x\in\mathcal{V}}m(x)V(x)u^2(x)<\infty \right \},
\end{eqnarray*}
which describes a ``maximal'' closed extension of $q_0$. We consider a closed form $q$ with domain $E=D(q)$ associated to the graph with
\[
D(q^{(D)}) \subset D(q) \subset D(q^{(N)}) \quad \text{ and } q= q^{(N)} \text{ on } D(q)
\]
with associated operator $L$. 
$E$ is a Hilbert space with the inner product
\begin{displaymath}
(u,v)_{E}=\dfrac{1}{2}\sum\limits_{x,y\in\mathcal{V}, x \sim y}b(x,y)(u(y)-u(x))(v(y)-v(x)))+\sum\limits_{x\in\mathcal{V}}m(x)V(x) u(x)v(x),\nonumber
\end{displaymath}
for each $u,v\in E$ with induced norm $\| u\|_E = q(u)^{1/2}$. 
Now, we formulate the assumptions on the potential that will guarantee the spectrum's discreteness for the operator $L$. Before, we need to define a concept, which guarantees the imbedding $E\hookrightarrow \ell^\infty(\mathcal V)$.

\begin{definition}\label{cancompsubset}
We say that a subset $\mathcal K$ of $\mathcal V$  is a canonically compactifiable subset if
\begin{align*}
r_{\mathcal K}:  u &\mapsto \chi_{\mathcal K} u \\
D(q^{(N)})&\to \ell^\infty,
\end{align*}
is a continuous operator, where $\chi_{\mathcal K}$ denotes the characteristic function on $\mathcal K$. We write in this case $\mathcal K \Subset \mathcal V$.  
\end{definition}


This generalizes a concept introduced in \cite{ke-le 2} (adapted as in \cite{HKSW23}):
\begin{definition}\label{cancompgraph}
    The graph $\Gamma=(\mathcal V, b, c)$ is canonically compactifiable if there exists a continuous imbedding
    \begin{gather*}
        D(q^{(N)}) \hookrightarrow  \ell^\infty.
    \end{gather*}
\end{definition}

\begin{remark}
    For a canonically compactifiable graph $\Gamma=(\mathcal V, b, c)$ we have in particular the continuous imbedding
    $$D(q) \hookrightarrow  \ell^\infty.$$ 
    Throughout the work we consider canonical compactifiability as an assumption and derive sufficient conditions in Section~\ref{spectraltheorysection}. However, the existence results that we develop continues to hold under weaker conditions that guarantee the imbedding $D(q) \hookrightarrow \ell^\infty$.  We would like to mention that in \cite{HKSW23} sufficient conditions for such $\ell^\infty$-inequalities were investigated for example in \cite{HKSW23}. 
\end{remark}

Throughout the paper, we assume: 

\begin{assumption}[growth assumption on the potential]\label{ass:pot1} We assume 
\begin{equation}\label{eq:infv}
     \sup\limits_{\mathcal{K} \Subset \mathcal{V}, \, m(\mathcal{K}) < \infty}\inf\limits_{x\in \mathcal{V}\setminus\mathcal{K}}V(x)= \infty.
     \end{equation}
\end{assumption}

\begin{remark}
A necessary and sufficient condition for the discreteness of the spectrum was shown for example in \cite[Theorem~20]{ke-le}. In \cite[Lemma~2.2]{grig-2} conditions for canonical compactifiability were considered as well as for the discreteness of the spectrum of the discrete Laplacian under stricter assumptions than the ones considered here.
\end{remark}

We will see in Section~\ref{sobolevsection} that Assumption~\ref{ass:pot1} guarantees the discreteness of the spectrum of $\sigma(L)$. We denote the eigenvalues of $L$ by 
\begin{equation}
   0\le \lambda_1 \le \lambda_2 \le \cdots, 
\end{equation}
where the eigenvalues are counted with their multiplicities, meaning that any eigenvalue appears as many times as its algebraic multiplicity indicates.  Let us define the (closed) subspaces generated by the eigenvectors with eigenvalues $<\lambda$, $=\lambda$ and
$>\lambda$ which are denoted by $E^-$, $E^0$ and
$E^+$, respectively:
\begin{equation}\label{eq:energydecomp}
E^-=\bigoplus\limits_{\{n:\lambda_n<\lambda\}} \ker(L-\lambda_n I),\quad E^0= \ker(L-\lambda I), \quad E^+=\bigoplus\limits_{\{n:\lambda_n>\lambda\}} \ker(L-\lambda_n I).
\end{equation}
By the spectral decomposition theorem we have $E= E^- \oplus E^0 \oplus E^+$.

\subsection{Formulation of the problem}

In the present paper, we consider the discrete NLS equations 
\begin{equation}\label{eq:NLS}
-\Delta u(x)+V(x)u(x)-\lambda u(x)=\kappa f(x,u(x)),\;\;\;x\in\mathcal{V},
\end{equation}
where $u$ is a real-valued function on $\mathcal{V}$, $\kappa=1$ (self-focusing) or $\kappa=-1$ (defocusing), 
$\lambda$ is a real parameter, and the Laplacian  is defined  by 
\begin{equation}\label{laplacian}
\Delta u(x)=\dfrac{1}{m(x)}\sum\limits_{y \sim x}b(x,y)(u(y)-u(x)).
\end{equation}

 We also assume: 
\begin{assumption}[Assumptions on the nonlinearity] \label{ass:nonlinearity}
We assume
\begin{itemize}\em
\item[$(f_1)$] $u\mapsto f(x,u)$ is a measurable, continuous function and $f(x,0)=0$ for all $x\in \mathcal V$.	
	\item[$(f_2)$]For all $x\in\mathcal{V}$,
	\begin{equation*}\label{eqn(f_1)}
	|f(x,u)|\leq \mu(R)|u|
	\end{equation*}
	whenever $|u|\leq R$, where $\mu(R)$ is non-decreasing, $\mu(R)>0$ if $R>0$, and $\mu(R)\to 0=\mu(0)$ as 
	$R\to 0$.
	
	\item[$(f_3)$] The function $f(x,u)/|u|$ (extended by $0$ to $u=0$) is strictly increasing.
	
	\item[$(f_4)$] $F(x,u)/u^2\to \infty$ as $|u|\to\infty$ for all $x\in\mathcal{V}$, where
	$$
	F(x,u)=\int_0^uf(x,s) ds.
	$$
\end{itemize}
\end{assumption}

Our goal is to use the method in \cite{sz-weth-10} under our assumptions to prove the existence of solutions of \eqref{eq:NLS}. Let us emphasize the flexibility in the approach. The method allows the treatment of the focusing and defocusing case and is not limited to the case, when local minimizers exist.

We will study \eqref{eq:NLS} via the critical points of the functionals 
\begin{equation}\label{energy}
J_\lambda(u)=\frac{1}{2}q_\lambda(u)-\kappa \sum_{x\in \mathcal{V}}m(x)F(x,u(x))\,,
\end{equation}
where the quadratic form $q_{\lambda}$ is defined by
\begin{equation*}
q_{\lambda}(u)=q(u)-\lambda(u,u)_m=q(u)-\lambda\|u\|^2_{\ell^2_m}\,
\end{equation*}
on the space $E$.   \\
\begin{proposition}\label{prop_for lemma}
    Let $\Gamma = (\mathcal V, b, c)$ be canonically compactifiable, then the functional $J_{\lambda}$ is of class $C^1$ on the energy space $E$ and its derivative $\nabla J_{\lambda}(u)$, $u\in E$, as a linear functional on $E$, is given by 
\begin{align*}
\nabla J_\lambda(u)v&=q_\lambda(u,v)-\kappa\sum_{x\in\mathcal{V}}m(x) f(x, u(x))v(x)\\
&= \sum_{x\in \mathcal{V}} \left (\Delta u(x) + V(x) u(x)\right ) v(x) - \kappa\sum_{x\in\mathcal{V}}m(x) f(x, u(x))v(x) \,,\quad \forall v\in E\,
\end{align*}
via $\kappa=1$ and $\kappa=-1$.

\end{proposition}
An immediate consequence is that the critical points of $J_\lambda$ characterize the solutions to \eqref{eq:NLS}.
\begin{corollary}
    Let $\Gamma = (\mathcal V, b, c)$ be canonically compactifiable. $u\in E$ is a critical point of $J_\lambda$, i.e. $\nabla J_\lambda(u)=0$ if and only if $u\in E$ solves \eqref{eq:NLS}.
\end{corollary}

\subsection{Main results}
We can now introduce our main results on the existence and bifurcation of solutions of \eqref{eq:NLS}:
\begin{theorem}\label{t5.1}
Let $\Gamma = (\mathcal V, b, c)$ be canonically compactifiable.
	\begin{itemize}
		\item[$(a)$] For the case $\kappa=+1$, the  problem  \eqref{eq:NLS} has a
		nontrivial solution $u\in E$. If, in addition, $f(x,s)$ is odd with respect to $s$, then there
		exist infinitely many pairs of nontrivial solutions.
		
		\item[$(b)$] For the case $\kappa=-1$ and $\lambda>\lambda_1$, there exists a nontrivial solution of  the  problem  \eqref{eq:NLS} in
		$E$.
		If, in addition, the nonlinearity is odd and $\lambda>\lambda_n$, then the problem has at least $N$
		pairs of nontrivial solutions, where
		$$
		N=\sum_{k=1}^{n}\dim \ker(L-\lambda_k I)\,.
		$$
		\item[$(c)$] For the case $\kappa=-1$ and $\lambda\leq \lambda_1$, the  problem  \eqref{eq:NLS} has no nontrivial
		solution
		in $E$.
	\end{itemize}
\end{theorem}

\begin{remark}\label{existenceofcp}
We apply the critical point theory for $J_{\lambda}$ in order to prove the existence of critical points (see Appendix~\ref{s7}). Let us now very briefly summarize the approach. If we define $F= E^- \oplus E^0$ and for $u\in E\setminus F$ we consider the minimax problem
        \begin{equation}\label{eq:clambda}
        c_{\lambda}:=c= \inf_{w\in {E\setminus F} }\max_{\substack{u=v+tw\\ v\in F,\,t\in \mathbb R}} J_{\lambda}(u).
    \end{equation}
    Under the assumptions of Theorem~\ref{t5.1}, we will construct a critical point $u_\lambda \in E$ of $J_\lambda$ for which $J_\lambda(u_\lambda)= c_\lambda$ attains the critical level (see Theorem~\ref{t3.2} and Remark~\ref{r3.3}). We henceforth refer to $u_\lambda$ as the \emph{ground critical point} of $J_\lambda$.
\end{remark}
We continue the section with a result, which is an immediate consequence of the better Sobolev imbedding, which we will discuss in Proposition~\ref{prop:sobpimb}.

\begin{corollary}
Under the assumptions of Theorem~\ref{t5.1}, if 
\begin{equation}\label{eq:replacesexpgrowth}
    \inf_{\mathcal K\subset \mathcal V,\, m(\mathcal K) < \infty} \sum_{x\in V\setminus \mathcal K} \frac{(m(x))^2}{m(x)+c(x)} < \infty
\end{equation}
 holds, then the solutions of \eqref{eq:NLS} are in $\ell^p_m$ for all $p\in [1,\infty]$.
\end{corollary}
\begin{remark}
Note that only the case $p\in[1,2)$ requires the additional assumption \eqref{eq:replacesexpgrowth} (see Section~\ref{subsec:interpolation}).
\end{remark}

\begin{remark}
    In \cite{Mmatrix}, the authors introduced an Agmon-type distance function governing the decay of eigenfunctions to discrete eigenvalue problems. In particular, decay estimates for the eigenvectors of the discrete Schrödinger operator (c.f. \cite{ste-23}). 

    In this article, we will not investigate the decay of solutions further. However, let us highlight how decay estimates can be used to show decay rates for the eigenfunctions of Schrödinger operators. Let us define for $u\in D(L)$
    $$
    V_0(x):= \frac{\kappa f(x,u(x))}{u(x)},
    $$
    then by $(f_2)$ we have $V_0 \in \ell^\infty$. Then an immediate consequence of the Kato--Rellich theorem is that $\mathfrak L = L+ V_0$ is a relatively compact perturbation of the operator $L$ and the spectrum $\sigma(\mathfrak L)$ is purely discrete, provided that $\sigma(L)$ is discrete. if $u\in D(L)$ is a solution of \eqref{eq:NLS}, then it is an eigenfunction of $\mathfrak L$. In particular, such decay estimates will be inherited for the solutions of \eqref{eq:NLS}.
 \end{remark}


  In the next result, we discuss the behavior of solutions as $\lambda\not \in \sigma(L)$ approaches an eigenvalue. In this context, we will obtain estimates on $\| u_{\lambda} \|_E$ for the ground critical points $u_{\lambda}$ in \eqref{energy} depending on the distance to the spectrum $$\delta(\lambda):= \operatorname{dist}( \lambda, \sigma(L)).$$

\begin{theorem}\label{theorem2}
        Let $\Gamma = (\mathcal V, b, c)$ be canonically compactifiable. Suppose for some $q>2$, $p\ge 2$ and $a_0, a_1 >0$ that the nonlinearity satisfies
		\begin{align}\label{AR}
      0<qF(x,s)&\leq f(x,s)s,\;s\in \mathbb{R}\setminus \{0\}\\
            \label{6.1}
		F(x,s)&\geq a_0|s|^p\\
		\label{6.3}
		|f(x,s)|&\leq a_1|s|^{p-1}.
		\end{align}
        
        Let $u_\lambda$ be the ground critical point of \eqref{energy}, then:
        \begin{enumerate}[a)]
        \item If $\lambda < \lambda_1$ and $\kappa=1$, then
        there exists a constant $C>0$ such that 
 		\begin{equation}\label{eq:bifurcationfirst}
 		    \|u_{\lambda}\|_E\leq C(\lambda_{1}-\lambda)^{\frac{1}{p-2}}
 		\end{equation}
        \item If $\lambda \in (\lambda_{k-1}, \lambda_k)$ for $k>1$. 
        \begin{itemize} \item Suppose 
        $\delta(\lambda)=\lambda_{k}- \lambda$ and $\kappa=1$, then there exists a constant $C>0$ such that \begin{equation}\label{eq:bifurcation2}	
         \|u_{\lambda}\|_E\leq C(\lambda_k-\lambda
         )^{\frac{1}{p-2}}
         \end{equation}
         holds.
 		 
         \item Suppose $\delta(\lambda) = \lambda -  \lambda_{k-1}$ and $\kappa=-1$,  then there exists a constant $C>0$ such that
        \begin{equation}\label{eq:bifurcationsecond}	
         \|u_{\lambda}\|_E\leq C(\lambda-\lambda_{k-1})^{\frac{1}{p-2}}
        \end{equation}
        \end{itemize}
        \end{enumerate}
 		\end{theorem}

        \begin{remark} 
            Note that, one example that satisfies the assumptions on the nonlinearity $(f_1)$-$(f_4)$, \eqref{AR}, \eqref{6.1}, \eqref{6.3} is the function $f(x,s) = g(x) |s|^{p-2}s$, where $p>2$ and $g(x)\ge c$ for some $c>0$. Then 
            $$ F(x,s) = \frac{g(x)}{p} |s|^p$$
            and one easily verifies all the properties.
        \end{remark}
\section{Preliminaries: Spectral theory of Schrödinger operators}\label{spectraltheorysection}

\subsection{On discrete Sobolev inequalities}\label{sobolevsection} First we review a condition that guarantees the imbedding $E\hookrightarrow  \ell^{\infty}$ for a class of graphs, and in the next step find conditions, where the imbedding holds in the general setting. \\ 

Canonical compactifiability was studied extensively in \cite{ke-le 2} and can be related to a geometric condition, that is closely related to the diameter of a graph. A natural choice for a metric on $\Gamma =(\mathcal V, b,c)$ is given via
\begin{equation}
 d(x,y) = \inf \left \{ \sum_{i=1}^n \frac{1}{b(x_{i-1}, x_i)}:\, (x_0, \ldots, x_n) \text{ is a path from $x$ to $y$}\right \}.
\end{equation}
The diameter of a set $\mathcal K \subset \mathcal V$ is then defined via
\begin{equation}
\diam_d(\mathcal K):= \sup_{x,y\in \mathcal K} d(x,y).
\end{equation}

It was shown in \cite[Corollary~4.4]{ke-le 2}:
\begin{proposition}\label{connected}
      Let  $\Gamma=(\mathcal V, b, c)$ be connected. Then $\Gamma=(\mathcal V, b,c)$ is canonically compactifiable if\begin{equation}\label{eq:finitediameter}
        \diam_d(\Gamma) := \sup_{x,y\in \mathcal V} d(x,y)< \infty.  
      \end{equation}
    \end{proposition}
    \begin{remark}\label{rmk:otherdirectioncomp}
        Due to \cite[Theorem~4.3]{ke-le 2}, for connected graphs the condition \eqref{eq:finitediameter} is equivalent to  
        \begin{align*}
            \widetilde E = \{ f: \mathcal V \to \mathbb R|\, q(f) < \infty\} \subset \ell^\infty
        \end{align*}
        when $c\equiv 0$.  By \cite[Theorem~3.2]{LSS18}, \eqref{eq:finitediameter} holds if and only if a global Poincar\'e inequality holds, i.e. there exists a positive constant $c>0$ with
        \begin{equation}
            \|f\|_{\mathcal V}^2 := \sup_{x\in \mathcal V} f(x) - \inf_{x\in \mathcal V} f(x) \le c\; q(f)
        \end{equation}
    for all $f\in \widetilde E$.

        Furthermore, even when $c\not \equiv 0$ there exists a metric $\sigma$ such that   
        \(
            \widetilde E \subset \ell^\infty
        \) 
        if and only if
        \begin{equation}
            \diam_\sigma(\Gamma) := \inf_{x,y\in \mathcal V} \sigma(x,y)< \infty.  
        \end{equation}

        Canonical compactifiability is hence strongly related to the geometry of the graph.
    \end{remark}

    We introduce an additional concept to generalize the spectral theory on graphs with infinite measure:

\begin{lemma}\label{lem:compactsubset}
   Let  $\Gamma=(\mathcal V, b, c)$ be a connected graph and $\mathcal K \subset \mathcal V$. Suppose $\diam_d(\mathcal{K}) < \infty$, then $\mathcal K \Subset \mathcal V$. \footnote{Recall the notation from Definition~\ref{cancompsubset}}
\end{lemma}
\begin{proof}
    Suppose $\diam(\mathcal{K}) < \infty$, then for $f\in D(q^{(N)})$ and any path $$x=x_0,\; x_1,\;\ldots,\; x_n=y$$ with $x,y\in \mathcal K$, we have
    \begin{align*}
         f(x) - f(y) &= \sum_{j=1}^n f(x_j) - f(x_{j-1}) \\
         &\le \left (\sum_{j=1}^n \frac{1}{b(x_j, x_{j-1})} \right )^{1/2}  \left ( \sum_{j=1}^n b(x_j, x_{j-1}) | f(x_{j}) - f(x_{j-1})|^2 \right )^{1/2}.
    \end{align*}
    In particular, 
    \begin{equation}
        \sup_{x\in \mathcal K} f(x) - \inf_{y\in \mathcal K} f(y) \le \diam_d(\mathcal K)^{1/2}  q^{(N)}(f)^{1/2}
    \end{equation}
and we have $\|\chi_{\mathcal K} f\|_\infty \le C\; q^{(N)}(f)^{1/2}$ for some $C>0$. In particular, $r_{\mathcal K}: D(q^{(N)}) \to \ell^\infty$ is bounded.
\end{proof}

\begin{corollary}\label{cor:summability}
        Let $\mathcal K \subset \mathcal V$ is a connected subset of $\Gamma$.  Then $K\Subset \mathcal V$ if
        \begin{equation}
            \sum_{x,y\in \mathcal K, \, b(x,y)>0} \frac{1}{b(x,y)} < \infty.
        \end{equation}
    \end{corollary}
    \begin{proof}
        Then $\diam_d(\mathcal K)<\infty$ and we have
        \begin{align*}
            \sup_{x\in \mathcal K} f(x) - \inf_{y\in \mathcal K} f(y) &\le \diam_d(\mathcal K)^{1/2} q^{(N)}(f)^{1/2}\\ &\le \left ( \sum_{x,y\in \mathcal K, \, b(x,y)>0} \frac{1}{b(x,y)}\right ) ^{1/2} q^{(N)}(f)^{1/2} . 
        \end{align*}
        In particular, we have $\|f\|_\infty \le C \; q^{(N)}(f)^{1/2}$ for some $C>0$.
    \end{proof}

\begin{remark}
It is easy to see that under our assumption any finite set is a canonically compactifiable subset. It is immediate then that $r_{\mathcal K}(D(q^{(N)})) \to \ell^\infty$ is continuous.
\end{remark}

\begin{proposition}\label{lem:sobolev1}
Let $\Gamma=(\mathcal V, b, c)$ be a combinatorial graph and suppose 
\begin{equation}\label{Sobolev_Assumption}
   \sup_{\mathcal{K}\subset  \mathcal{V}:\, \diam_d(\mathcal K) <\infty} \inf_{x\in \mathcal{V}\setminus \mathcal{K}} (c(x) + m(x) )>0,
\end{equation}
then $D(q^{(N)})\hookrightarrow \ell^\infty$.
\end{proposition}
\begin{proof}
We separate the proof into two cases.
\begin{itemize}
		\item[$(i)$] If $\diam_d{(\mathcal{V})}<\infty,$ then we do not need any other assumption. From Proposition~\ref{connected}, we can say that the graph is canonically compactifiable. This implies a continuous imbedding $D(q^{(N)}) \hookrightarrow \ell^\infty$ by Lemma~\ref{connected}.

		\item[$(ii)$] If $\diam(\mathcal{V})$ is not finite, then by  \eqref{Sobolev_Assumption}, there exists $\tilde{c}>0$ such that 
$$\inf_{x\in \mathcal{V}\setminus \mathcal{K}} (c(x) + m(x))\geq \tilde{c},$$
for a subset $\mathcal{K}\subset \mathcal{V}$ with $\diam_{d}(\mathcal K) < \infty$, which implies
\begin{equation}\label{canonically1}
    c(x) + m(x)\geq \tilde{c}
\end{equation}
for all $x\in \mathcal{V}\setminus \mathcal{K}$. On the other hand,
$$q^{(N)}(u)\geq (c(x) + m(x))|u(x)|^2,$$ for all $u\in D(q^{(N)})$ and $x\in \mathcal{K}$. Together with this and \eqref{canonically1}, we can obtain

\begin{eqnarray}\label{outside}
  \dfrac{1}{\tilde{c}} q^{(N)}(u)   &\geq  \dfrac{1}{\tilde{c}}\left(c(x)+m(x)\right)|u(x)|^2 \geq|u(x)|^2 ,
 \end{eqnarray}
 for all $x\in \mathcal{V}\setminus \mathcal{K}$. Then $\mathcal K \Subset \mathcal V$ by Lemma~\ref{lem:compactsubset} and there exists $\tilde C >0$ such that
 \begin{equation}\label{inside}
 |u(x)|^2 \leq \dfrac{1}{\tilde{C}} q^{(N)}(u)
 \end{equation}
 for all $x\in \mathcal {K}$. Thus, combining \eqref{outside} and \eqref{inside} we have the required imbedding $D(Q^{(N)}) \hookrightarrow \ell^\infty$. 
 \end{itemize}        
 \end{proof}

 \subsection{On interpolation inequalities and better Sobolev imbeddings} \label{subsec:interpolation}
 If $E\hookrightarrow \ell^\infty$, then for $p \in [2,\infty]$ we have via interpolation $$E\hookrightarrow \ell_m^2\cap \ell^\infty \hookrightarrow \ell^p_m.$$

 A sufficient condition for the imbedding $E\hookrightarrow \ell^\infty$ was derived in Proposition~\ref{lem:sobolev1}, and we will derive in this section a stronger assumption that will guarantee the better imbedding $$E\hookrightarrow \ell_m^1 \cap \ell^\infty \hookrightarrow \ell^p_m$$ for all $p\in [1,\infty]$. Such an imbedding was for example obtained in \cite[Lemma~2.1]{grig-2} under the assumption $\tfrac{1}{V} \in \ell^1_m$. We will, however, generalize the idea to adapt it in our context under weaker conditions.

\begin{lemma}\label{lem:lem1imbed}
Suppose \eqref{eq:replacesexpgrowth} holds, 
then $E\hookrightarrow \ell^1_m$.
\end{lemma}
\begin{proof}
Let $\mathcal K\subset \mathcal V,\, m(\mathcal K) < \infty $ such that $$\sum\limits_{x\in \mathcal{V}\setminus\mathcal{K}}\dfrac{(m(x))^2}{c(x)+m(x)}< \infty.$$ Then for $u\in E$ we have
     \begin{align*}
			\sum\limits_{x\in \mathcal{V}\setminus \mathcal{K}}m(x)|u(x)| &= \sum\limits_{x\in \mathcal{V}\setminus \mathcal{K}} m(x)
            \;\dfrac{1}{{V(x)}^{1/2}} {{V(x)}^{1/2}}|u(x)|\\ &\leq   \left(\sum\limits_{x\in \mathcal{V}\setminus \mathcal{K}}m(x)\dfrac{1}{V(x)} \right)^{1/2}\left(\sum\limits_{x\in \mathcal{V}\setminus \mathcal{K}} m(x)V(x)|u(x)|^2\right)^{1/2}.\\
            &\le \left(\sum\limits_{x\in \mathcal{V}\setminus \mathcal{K}}\dfrac{(m(x))^2}{c(x)+m(x)} \right)^{1/2}\|u\|_E.
		\end{align*} 
Furthermore, since  $\mathcal K\subset \mathcal V,\, m(\mathcal K) < \infty$, we have 
    \begin{displaymath}
        \sum_{x\in \mathcal K } m(x) u(x) \le m(\mathcal K)^{1/2} \left(\sum\limits_{x\in  \mathcal{K}}m(x)|u(x)|^2 \right)^{1/2} \le m(\mathcal K)^{1/2} \| u\|_E
    \end{displaymath}
and we conclude $E\hookrightarrow \ell_m^1$.
\end{proof}

    

\begin{proposition}\label{prop:sobpimb}
Let $\Gamma=(\mathcal V, b,c)$ be canonically compactifiable. Suppose \eqref{eq:replacesexpgrowth} holds,
then $E\hookrightarrow\ell^p_m$ for all $p\in [1,\infty]$.
\end{proposition}
\begin{proof}
 By our assumptions, together with Lemma~\ref{lem:lem1imbed} we have $E\hookrightarrow \ell^{\infty}\cap \ell^1_m$. Then by interpolation we have 
 $$ E\hookrightarrow \ell^p_m$$
 for all $p\in [1,\infty]$.
\end{proof}

\subsection{On compact imbeddings and discreteness of spectrum}

The concept of canonical compactifiability was previously used in \cite{ke-le 2} in order to discuss the discreteness of the spectrum of the Laplacian. \\

 First we will prove that if $m(\mathcal V) < \infty$, then we have a compact imbedding $ E \hookrightarrow \ell^2$.
\begin{lemma}\label{lem:l2imbedding}
Let $\mathcal K \subset \mathcal V$ be a subset with $m(\mathcal K) < \infty$, then $r_{\mathcal K}:\ell^\infty \hookrightarrow \ell^p_m$ defined via
\begin{equation*}
(r_{\mathcal K}f)(x)=\begin{cases}
    f(x), &\quad x\in \mathcal K, \\
    0, &\quad x \not \in \mathcal K
\end{cases}
\end{equation*}
is compact  for all $p\in [1,\infty)$.
\end{lemma}

\begin{proof}
Denote by $\mathcal K^n$ be a sequence of finite subsets with $\mathcal K = \bigcup_{n\in \mathbb N} \mathcal K^n$ and let $r_{\mathcal K^n}: \ell^\infty \to \ell_{m}^p$ be the restriction operator 
\begin{equation}
(r_{\mathcal K^n}f)(x)=\begin{cases}
    f(x), &\quad x\in \mathcal K^n, \\
    0, &\quad x \not \in \mathcal K^n
\end{cases}
\end{equation}
Then each $r_{\mathcal K^n}$ is finite rank and the sequence $(r_{\mathcal K^n})_{n\in \mathbb N}$ converges to the imbedding $\iota: \ell^\infty \to \ell^p_m$ in operator norm, thus $\iota$ is compact.
\end{proof}


 \begin{proposition} \label{lem:sobolev2}  $E \hookrightarrow \ell^2_m$ is compact. In particular, $L$ has pure discrete spectrum.
 \end{proposition}

\begin{proof}
Let us prove it in the same way. 
\begin{itemize}
		\item[$(i)$] If $m(\mathcal{V})<\infty$ and $\Gamma$ is canonically compactifiable, then due to canonical compactifiability and Lemma~\ref{lem:l2imbedding}, the imbedding from $E \hookrightarrow \ell^2_m$ is compact as a composition of a continuous and compact imbedding. Thus, $L$ has compact resolvent and by standard results, $L$ has pure discrete spectrum. 
		
		\item[$(ii)$] Otherwise, consider an exhausting sequence of canonically compactifiable subsets $\mathcal{K}_t\Subset \mathcal{V}$ with $m(\mathcal K_t) < \infty$ satisfying 
  $$\mathcal{K}_t\subset \mathcal{K}_{t+1} \text{\;\;and\;\;} \inf\limits_{x\in \mathcal{V}\setminus\mathcal{K}_t}V(x)\to \infty \text{\;as\;}t\to \infty,$$ which exists due to Assumption~\ref{ass:pot1}. Since $\mathcal V$ is countable, we can introduce a list of vertices 
  \begin{displaymath}
  \mathcal V=\{ v_1, v_2, v_3, \ldots\}.    
  \end{displaymath}
  Then in order to guarantee $\bigcup\limits_{t\in \mathbb{N}} \mathcal{K}_t=\mathcal{V}$, we can add $\{v_1,v_1,\ldots,v_t\} $ to $\mathcal{K}_t$ which does not change the conditions
  $$m(\mathcal{K}_t)<\infty  \text{\;\;and\;\;} \inf\limits_{x\in \mathcal{V}\setminus\mathcal{K}_t}V(x)\to \infty \text{\;as\;}t\to \infty.$$ We want to show that if $\{u_n\}$ is a bounded sequence of functions in $E$, then there exists a convergent subsequence in $\ell^2_m$. 
        

Let us take a bounded sequence $\{u_n\}$ in $E$ such that for some $c>0$ we have
 $$\|u_n\|_{\ell^2_m}^2\leq \| u_n \|_E^2 \leq c,\;\;\forall n\in \mathbb{N}.$$

Let $\mathcal K$ be a canonically compactifiable subset with $m(\mathcal K) < \infty$ and
\begin{equation}
(r_{\mathcal K} f)(x) = \begin{cases} f(x), &\quad x\in \mathcal K\\ 0, &\quad x\not \in \mathcal K. \end{cases}
\end{equation}
 Then $r_{\mathcal K}:E\to \ell^2_m$ is compact by Lemma~\ref{lem:l2imbedding} and that there exists $u\in E$ such that
\begin{align*}
u_n &\rightharpoonup u \qquad\text{in } E \\
r_{\mathcal K} u_n &\to r_{\mathcal K} u \qquad \text{in } \ell^2_m 
\end{align*}
 
  To show that $u_{n_k}\to u$ in $\ell^2_m$ passing to a subsequence, let us take such a sequence $\{\mathcal{K}_t\}$ of canonically compactifiable subsets satisfying 
$$\chi_{\mathcal{K}_t} u_{n_k}\to \chi_{\mathcal{K}_t} u \text{\; in } \ell^2_m \text{\;as\;}t\to \infty,$$ or, there exists a subsequence $(u_{n_{k,t}})$ such that
$$\|\chi_{\mathcal{K}_t} u_{n_{k,t}}-\chi_{\mathcal{K}_t} u\|_{\ell^2_m}\leq \dfrac{1}{t},$$
that is, $\chi_{\mathcal{K}_t} u_{n_{k,t}}\to \chi_{\mathcal{K}_t} u$ in 
 $\ell^2_m \text{\;as\;}t\to \infty$. Then, 
 \begin{align*}
			c\geq \|u_{n_{k,t}}\|^2_E &\geq \sum\limits_{x\in \mathcal{V}} (c(x)+m(x))|u_{n_{k,t}}(x)|^2\\ &\geq   \sum\limits_{x\in \mathcal{V}\setminus \mathcal{K}_t}(c(x)+m(x))|u_{n_{k,t}}(x)|^2\\ &\geq \inf\limits_{x\in \mathcal{V}\setminus\mathcal{K}_t}V(x) \sum\limits_{x\in \mathcal{V}\setminus \mathcal{K}_t}m(x)|u_{n_{k,t}}(x)|^2,
		\end{align*} 
        which implies 
        $$\dfrac{c}{\inf\limits_{x\in \mathcal{V}\setminus\mathcal{K}_t}V(x)}\geq \sum\limits_{x\in \mathcal{V}\setminus \mathcal{K}_t}m(x)|u_{n_{k,t}}(x)|^2= 1-\|\chi_{\mathcal{K}_t} u_{n_{k,t}}\|_{\ell^2_m}^2,$$ and we get
        $$\lim\limits_{t\to \infty}(1-\|\chi_{\mathcal{K}_t} u_{n_{k,t}}\|_{\ell^2_m}^2)\leq \lim\limits_{t\to \infty} \dfrac{c}{\inf\limits_{x\in \mathcal{V}\setminus\mathcal{K}_t}V(x)}=0.$$
By monotone convergence,
\begin{equation*}
    \|u\|_{\ell^2_m} = \lim_{t\to \infty} \|\chi_{\mathcal K_t} u\|_{\ell^2_m}, 
\end{equation*}
and we conclude
$$1-\|u\|_{\ell^2_m}^2=0, \text{\;thus,\;}\|u\|_{\ell^2_m}^2=1,$$ hence,  $u_n \to u$ in $\ell^2_m$. 
\end{itemize}        

 \end{proof}

 \begin{remark}\label{lem:sobolev}
    If $\Gamma=(\mathcal V, b, c)$ is canonically compactifiable and $m(\mathcal V) < \infty$, then Assumption~\ref{ass:pot1} becomes obsolete and Proposition~\ref{lem:sobolev2} remains true.
\end{remark}

\begin{figure}[ht]
    \centering
    \includegraphics{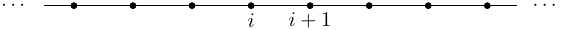}
    \caption{Visualization of the line graph in \autoref{ex:bettercondition}}
\end{figure}
 \begin{example}\label{ex:bettercondition} 
 Let us consider our result on the compact imbedding $E\hookrightarrow \ell^2_m$ with known results when $|V(x)|\to \infty$ in the following sense:
     \begin{equation}\label{eq:vinfty}
        \sup\limits_{\mathcal{K}\subset \mathcal V \,\text{finite}}\inf\limits_{x\in \mathcal{V}\setminus\mathcal{K}}V(x)= \infty,
    \end{equation}
    then the potential satisfies Assumption~\ref{ass:pot1} 
    and it is in fact known that such a condition implies discreteness of the spectrum of $L$. In this way, our results can be seen as a natural extension to this setting (see e.g. \cite[Corollary 4.19]{KeLeWo}).

    However, we can construct an example for which Assumption~\ref{ass:pot1} holds, but \eqref{eq:vinfty} is not satisfied. Suppose we have a line graph $\Gamma$ such that  
    \begin{equation*}
        m(i) = \begin{cases} \frac{1}{i+1}, &\quad i \ge 0 \\ 
        \frac{1}{i^2}, &\quad i <0.
        \end{cases}
    \end{equation*}
    \begin{equation*}
       b(i, i+1) = \begin{cases} i+1, &\quad i\ge 0 \\ i^2, &\quad i < 0.\end{cases}
    \end{equation*}
    with potential 
    \begin{equation*}
    V(i) = \begin{cases} i, &\quad i \ge 0 \\ 0, &\quad i < 0.\end{cases}
    \end{equation*}
     Then
    \begin{align*}
        \sup\limits_{\mathcal{K}\subset \mathcal V\, \text{finite}}\inf\limits_{x\in \mathcal{V}\setminus\mathcal{K}}V(x)= 0
    \end{align*}
    and \eqref{eq:vinfty} is not satisfied.
    Note that the graph $\Gamma=(\mathbb Z, b,c)$ does not have finite measure and does not satisfy $\diam_d(\Gamma) < \infty$. In particular the conditions in Proposition~\ref{connected} are not satisfied and we need to verify Assumption~\ref{ass:pot1} (see Remark~\ref{lem:sobolev}). Let us show that $\Gamma$ is still canonically compactifiable and $L$ has discrete spectrum.
    
    One easily verifies
    \begin{equation*}
        m(-\mathbb N) = \sum_{z<0} \frac{1}{z^2} < \infty
    \end{equation*}
    and
    \begin{equation*}
        \operatorname{diam}(-\mathbb N) = \sum_{z< 0} \frac{1}{b(z, z+1)} = \sum_{z<0} \frac1{z^2} < \infty.
    \end{equation*}
    In particular, $\mathbb K_n = -\mathbb N \cup \{0,1,\ldots, n\}$ is a canonically compactifiable subset and we have
    \begin{align*}
        \inf_{x\in \mathbb Z \setminus \mathcal K_n} V(x) &= n+1\\
        \inf_{x\in \mathbb Z \setminus \mathcal K_n} c(x) &= \frac{n+1}{n+2}.
    \end{align*}
    Thus,
    \begin{align*}
        \sup_{\mathcal K \Subset \mathbb Z} \inf_{x\in \mathbb Z \setminus \mathcal K} V(x) &\ge \lim_{n\to \infty}  \inf_{x\in \mathbb Z \setminus \mathcal K_n} V(x)= \infty\\
         \sup_{\mathcal K \Subset \mathbb Z} \inf_{x\in \mathbb Z \setminus \mathcal K} c(x) &\ge \lim_{n\to \infty}\inf_{x\in \mathbb Z \setminus \mathcal K_n} c(x) = 1>0.
    \end{align*}
    In particular, by Proposition~\ref{lem:sobolev1} we have $D(Q^{(N)}) \hookrightarrow  \ell^\infty(\mathbb Z)$ and by Lemma~\ref{lem:sobolev2} we have compact imbedding $E \hookrightarrow\ell^2_m(\mathbb Z)$. This verifies that $\Gamma$ is canonically compactifiable and $L$ has discrete spectrum.
 \end{example}

\section{Proof of Proposition~\ref{prop_for lemma}}
To prove Proposition~\ref{prop_for lemma}, we need to investigate properties of the nonlinearity. One can see that similar properties hold in the continuous case as well (see \cite[Lemma~3.1]{akd-pank}).\\

 For this purpose we define inspired by \cite{d-m}, ${f}_{R}$ and ${F}_{R}$ via\\

\begin{equation}\label{f_tilde}
    {f}_{R}(x,s)=
\left\{ \begin{array}{cc}
& 
\begin{array}{cc}
\;f(x,s), & |s|< R \\
\dfrac{f(x,R)}{R}s, & s\geq R \\
\dfrac{f(x,-R)}{-R}s, & s\leq -R
\end{array}
\end{array}\right.
\end{equation} and 
\begin{equation}\label{F_tilde2}
{F}_{R}(x,u)=\int_0^{u} {f}_{R}(x,s) ds\,,\end{equation}
respectively.

\begin{lemma}\label{lem:lipschitzr}
    Suppose $u\in \ell^2_m$, then $f_R(\cdot, u(\cdot))\in \ell^2_m$ for all $R>0$ and $F_R(\cdot, u(\cdot)) \in \ell^1_m$.
\end{lemma}
\begin{proof}
With $(f_2)$ we have 
\begin{gather*}
|f_R(\cdot\;, s)| \le \mu(R) |s|,\qquad 
|{F}_{R}(x,s)|\leq \frac{\mu(R)}{2} |s|^2,
\end{gather*}
and we conclude
\begin{gather*}
\|f_R(\cdot\;, u(\cdot)\|_{\ell^2_m} \le \mu(R) \|u\|_{\ell^2_m}, \qquad
\| F_R(\cdot\; u(\cdot))\|_{\ell^1_m} \le \frac{\mu(R)}{2} \|u\|_{\ell^2_m}^2.
\end{gather*}
\end{proof}

Let us denote the open ball with radius $R$ with
\begin{equation}\label{eq:aball}
B_R:= \{ u\in E: \|u\|_E \le R\}.
\end{equation}

\begin{lemma}\label{lem:lipschitzres}
	Under the assumptions $(f_1)$ and $(f_2)$, for every $R>0$, there exists $R'>0$ such that  $\|u\|_{\infty}\leq R'$ for each $u\in B_{R}\subset E$ and $$f(x,u(x))=f_{R'}(x,u(x)),\; F(x,u(x))=F_{R'}(x,u(x)), \; \forall x\in \mathcal{V},$$ where $f(.\;,u(.))\in \ell^2_m$ and $F(.\;,u(.))\in \ell^1_m$.
\end{lemma}
\begin{proof}
 First part of the proof follows from \eqref{f_tilde} and \eqref{F_tilde2}. For the remaining part of the proof, by the continuity of the imbedding $E\hookrightarrow \ell^{\infty}$, if $u \in E$ with $\|u\|_{E} < R$  then there is $R{'}$ such that $\|u\|_{\ell^{\infty}} < R{'}$ and together with  $(f_2)$, we have 
\begin{eqnarray}
\|f(.\;,u(.)\|^2_{\ell^2_m}&=&\sum_{x\in \mathcal{V}}m(x)|f(x,u(x))|^2 \nonumber \\
&\leq&\sum_{x\in \mathcal{V}}m(x)C^2_{R{'}}|u(x)|^2=C^2_{R{'}}\|u\|^2_{\ell^2_m}<\infty ,\nonumber
\end{eqnarray}
whenever $u\in B_R$, where $C_{R{'}}=\max\{\mu(R{'}),|f(x,R^{'})|,|f(x,-R^{'})|\}$. Also,
\begin{eqnarray}
\|f(.\;,u(.)\|_{\ell^1_m}=\sum_{x\in \mathcal{V}}m(x)|f(x,u(x))|&\leq&\sum_{x\in \mathcal{V}}m(x)\int\limits_0^{u(x)}|f(x,s)|ds\nonumber \\&\leq&\sum_{x\in \mathcal{V}}m(x)|u(x)|\;\mu(R^{'})\;|u(x)|\nonumber\\&=&\mu(R^{'})\|u\|^2_{\ell^2_m}<\infty,\nonumber
\end{eqnarray}
which completes the proof.
\end{proof}
\begin{lemma}
Let $R>0$ and $\Psi_R: \ell^2_m \to \mathbb{R}$ be defined by
\begin{equation*}
	\Psi_{R}(u)=\sum\limits_{x\in\mathcal{V}}m(x)F_{R}(x,u(x))\,,\quad u \neq 0.
\end{equation*}
	If $(f_1)$ and $(f_2)$ are satisfied, then the functional $\Psi_R(u)$  is $C^1$ and
\begin{equation}\label{e3.7}
\nabla \Psi_R(u)h=\sum\limits_{x\in\mathcal{V}}m(x) f_R(x,u(x))h(x) \,,\quad \forall h \in E\,.
\end{equation}
\end{lemma}
\begin{proof}
    For an arbitrary $u\in \ell^2_m$, if we take a bounded linear functional $A:=\nabla \Psi_{R}(u)$ defined as 
$$Ah=\sum\limits_{x\in\mathcal{V}}m(x)f_{R}(x,u(x))h(x) \,,\quad \forall h \,$$ then one can obtain that
$$\lim\limits_{h\to 0}\dfrac{|\Psi_{R^{'}}(u+h)-\Psi_{R^{'}}(u)-Ah|}{\|h\|_{\ell^2_m}}=0.$$  Then $f(x,u)=\dfrac{\partial F}{\partial u}(x,u)$, and using $$\dfrac{d}{dt}F(x,u+th(x))=\dfrac{\partial F}{\partial u}(x,u+th(x)) h(x),$$ we obtain 
$$\int_0^1 \dfrac{\partial F}{\partial u}(x,u+th(x))h(x)dt=F(x,u+th(x))-F(x,u(x)),$$ and so
$$F(x,u+th(x))-F(x,u(x))=\int_0^1 f(x,u+th(x))dt.$$
Using the last equation above and the Mean Value Theorem, we get
\begin{eqnarray}
& &|\Psi_{R^{'}}(u+h)-\Psi_{R^{'}}(u)-Ah|\nonumber\\ \nonumber&=&|\sum\limits_{x\in \mathcal{V}}m(x)[\tilde F_{R^{'}}(x,(u+h)(x))-F_{R^{'}}(x,u(x))]-\sum\limits_{x\in \mathcal{V}}m(x)\tilde f_{R^{'}}(x,u(x))h(x)|\nonumber \\&=&|\sum\limits_{x\in \mathcal{V}}m(x)\left(
\int_0^1 (f(x,u(x)+th(x))-f(x,u(x)) )dt
\right)    h(x)|\nonumber\\ & \leq& \sum\limits_{x\in \mathcal{V}}m(x)\left|\left(
\int_0^1 (f(x,u(x)+th(x))-f(x,u(x)) )dt
\right) \right|   |h(x)|\nonumber\\&\leq& \sum\limits_{x\in \mathcal{V}}m(x)\left(
\int_0^1 \left |f(x,u(x)+th(x))-f(x,u(x)\right|^2dt
\right)^{1/2}  \|h\|^2_{\ell^2_m}.\nonumber
\end{eqnarray}
Then by $(f_2)$ we have 
\begin{eqnarray*}
    |f_R(x, u(x) + th(x)) - f_R(x, u(x))|^2 &\le& 2 |f_R(x, u(x) + th(x))|^2 + 2|f_R(x, u(x))|^2 \\
    &\le& 2\mu(R)^2 (|u(x) + th(x)|^2 + |u(x)|^2)
\end{eqnarray*}
and by dominated convergence we have 
\begin{multline}
\lim\limits_{h\to 0}\dfrac{|\Psi_{R^{'}}(u+h)-\Psi_{R^{'}}(u)-Ah|}{\|h\|_{\ell^2_m}}\nonumber\\ \leq\sum\limits_{x\in \mathcal{V}}m(x)\lim\limits_{h\to 0}  \left(\left(
\int_0^1 (f(x,u(x)+th(x))-f(x,u(x)) )dt
\right)^2\right)^{1/2} =0\nonumber,
\end{multline}
which implies that $\Psi_{R}$ in $C^1$ and
\begin{equation*}
\nabla \Psi_R(u)h=\sum\limits_{x\in\mathcal{V}}m(x) f_R(x,u(x))h(x) \,,\quad \forall h \in E\,.
\end{equation*}

\end{proof}

\begin{lemma}\label{l2.1} 
Let $\Psi:E \to \mathbb{R}$ be defined by
\begin{equation*}
	\Psi(u)=\sum\limits_{x\in\mathcal{V}}m(x) F(x,u(x))\,,\quad u \neq 0
\end{equation*}
	If $(f_1)$ and $(f_2)$ are satisfied, then the functional $\Psi(u)$  is of the class $C^1$ and  weakly continuous and
\begin{equation}
\nabla \Psi(u)h=\sum\limits_{x\in\mathcal{V}}m(x) f(x,u(x))h(x) \,,\quad \forall h \in E\,.
\end{equation}
Furthermore, $\nabla \Psi: E \to E'$ is completely continuous, i.e. every weakly compact subset is mapped to a compact subset.

\end{lemma}
\begin{proof}


Let $u\in E$ be arbitrary. Then there exists $R>0$ such that $u\in B_R$ as defined \eqref{eq:aball}. Then by Lemma~\ref{lem:lipschitzres} we have 
\begin{displaymath}
    \Psi(u) = \Psi_R(u)
\end{displaymath}
and we conclude $\Psi$ is $C^1$ with
\begin{align}
    \nabla \Psi(u) h = \sum_{x\in \mathcal V} m(x) f(x, u(x)) h.
\end{align}
Due to the compact imbedding $E\hookrightarrow \ell^2$ we have weak continuity. Furthermore, 

 One can obtain $\nabla \Psi(u)$ replacing $f$ by $f_{R}$ as follows
\begin{equation*}
\nabla \Psi(u)h=Ah=\sum\limits_{x\in\mathcal{V}}m(x) f_{R}(x,u(x))h(x) \,,\quad \forall h \in E\,.
\end{equation*}
 
 Let $R>0$ be arbitrary. We can then rewrite $\nabla \Psi: E \to E'$ as a composition of continuous operators.  
 
 Consider the imbedding 
 \begin{align*}i_1:B_R&\to \ell^2_m\\ g&\mapsto g ,\end{align*}
 the Nemytskii operator 
 \begin{align*}
     N_R:\ell^2_m&\to \ell^2_m\\
     u &\mapsto f(\cdot, u(\cdot)),
 \end{align*}  
 and the dual imbedding 
 \begin{align*}
 i_2:\ell^2_m&\to E'\\
 g&\mapsto \left (h \mapsto \sum_{x\in \mathcal V} m(x) g(x) h(x) \right ). 
 \end{align*}
 Then 
 \begin{align*}
 \nabla\Psi = i_2 \circ N_R \circ i_1
 \end{align*}
 is completely continuous since $i_1$ is compact by Proposition~\ref{lem:sobolev2}. 




\end{proof}

\begin{proof}[Proof of Proposition~\ref{prop_for lemma}:]
Since the quadratic part of $J_{\lambda}$ is continuously differentiable, from  Lemma~\ref{l2.1}, \begin{equation*}
J_\lambda(u)=\frac{1}{2}q_\lambda(u)-\kappa \Psi(u)\,,
\end{equation*}
is $C^1$, and its derivative as a linear functional is given by 
$$
\nabla J_\lambda(u)v=q_\lambda(u,v)-\kappa \nabla \Psi(u)v \,,\quad \forall v\in E\,
$$ with $\kappa=1$ and $\kappa=-1$.\\

\end{proof}


	

\section{Proof of Theorem \ref{t5.1}}\label{s3}

Now we consider the energy functional $J_\lambda (u)$ defined in  \eqref{energy}  and apply Theorem~\ref{t3.2} to
$$
\kappa J_\lambda (u)=\frac{\kappa}{2}q_\lambda(u)-\Psi(u)\, .
$$
In accordance with the notations of Section 6 (Appendix), in our case,  $Q(u)=\kappa q_\lambda(u)$ and $\Phi(u)=\Psi(u)$ and $\Psi(u)$ is given in Lemma~\ref{l2.1}. Note that the cases $\kappa=1$ and $\kappa=-1$ have certain differences. 

\begin{proof}[Proof of Theorem~\ref{t5.1}] Under our new assumptions, we give the proof of this theorem as an adaptation of the one in \cite{akd-pank} with a discrete setting. In order to apply Theorem~\ref{t3.2} to the functional   $\kappa J_\lambda (u)$ , we first prove that the assumptions (i) - (v) of Section~\ref{s7} are  satisfied.\\

$(v):$ It follows from Lemma~\ref{l2.1}.\\

$(i):$ Again by Lemma~\ref{l2.1}, $\Phi$ is weakly lower semicontinuous. Now we show that $(\ref{e3.1})$ holds, that is, the remaining part of $(i)$ is satisfied. Since
$f(x,s)>0$  for $s>0$ and  $f(x,s)<0$  for $s<0$, it is clear that $F(x,s)>0$ for
all $s\neq 0$. From the Assumption $(f_2)$, we obtain the following estimation immediately 
$$
F(x,s)<\int_0^s|\tau||s|^{-1}f(x,s) d\tau=\frac{1}{2}f(x,s)s\,,
$$
and hence, 
$$\Psi(s)=\sum_{x\in \mathcal{V}}m(x)F(x,s)<\frac{1}{2}\sum\limits_{x\in\mathcal{V}}m(x)f(x,s)s=\frac{1}{2}\nabla\Psi(s)s\;,\;\;\forall s\neq 0,
$$
Thus, $(\ref{e3.1})$ holds, that is, the remaining part of $(i)$ is
satisfied.\\

$(iii):$ In view of Assumption $(f_1)$, and the imbeddings $E\hookrightarrow \ell^\infty$ and $E\hookrightarrow \ell^2_m$, it is easy to see that if $u\in E$, then 
$$
\|f(x,u)\|^2_{\ell^2_m}\leq \mu^2(R)\|u\|^2_{\ell^2_m}\leq C\mu^2(R)\|u\|_E^2\,,
$$
where $R=\|u\|_{\ell^\infty_m}\leq C_1\|u\|_E$. We know that $\mu(R)\to 0$ as $R\to 0$. Combining this with
Lemma~\ref{l2.1}, we concluded that $(iii)$ holds.\\

$(iv):$ Suppose to the contrary that there is a weakly compact set
$W\subset E \setminus \{0\}$ and a sequence $\{u_n\}$ in $W$ such that $\tau_n^{-2}\Psi(\tau_n u_n )$ is bounded as a sequence $\tau_n\to \infty$. Passing to a subsequence, we can assume that $u_n\to
u\neq 0$ weakly in $E$ and, by the assumption $E\hookrightarrow \ell^2_m$ compactly, strongly in
$\ell^2_m$. Hence, passing to a further subsequence, $u_n(x)\to u(x)$ a.e. on $\mathcal{V}$.
From Assumption $(f_3)$, $\dfrac{F(x,\tau_{n}u_n)}{{(\tau_n
		u_n)}^2}\to \infty$ since $|\tau_nu_n(x)|\to \infty$.  So if $u(x)\neq 0$, since $F\geq 0$, a consequence of  the Fatou Lemma is
$$\frac{\Psi(\tau_nu_n)}{{\tau_n}^2}=\sum\limits_{x\in\mathcal{V}}m(x)\frac{F(x,\tau_n 
	u_n)}{{(\tau_n
		u_n)}^2}u_n^2 \to \infty\,, as\;\; n\to \infty,$$
we get a contradiction.

$(ii):$ For the real number $\lambda$, let us define the (closed) subspaces generated by the eigenvectors with eigenvalues $<\lambda$, $=\lambda$ and
$>\lambda$ which are denoted by $E^-$, $E^0$ and
$E^+$, respectively:
$$E^-=\bigoplus\limits_{\{k:\lambda_k<\lambda\}} \ker(L-\lambda_k I),\quad E^0= \ker(L-\lambda I), \quad E^+=\bigoplus\limits_{\{k:\lambda_k>\lambda\}} \ker(L-\lambda_k I),$$
The form $q_\lambda$ is positive (respectively, negative) definite on
$E^+$ (respectively, on $E^-$), \emph{i.e.}, there exists a constant $\beta=\beta(\lambda)>0$ such that
\begin{equation}\label{e4.1}
	\pm q_\lambda (u)\geq\beta\|u\|^2_E \,,\quad u\in E^\pm\,.
\end{equation}

Here $E^-$ and $E^0$ are finite dimensional subspaces, while $E^+$
has infinite dimension. If $\lambda_n<\lambda\leq\lambda_{n+1}$, then $E^-=\bigoplus_{k=1}^n \ker(L-\lambda I)$
and $N$ is the dimension of this space.
These subspaces serve the functional $ \kappa J_\lambda$. If $\kappa=1$, then $F=E^-\oplus
E^0$. If $\kappa=-1$, since the form $-q_\lambda$ is positive (respectively, negative) on $E^-$(respectively, on $E^+$), $F=E^+\oplus E^0$ in this case.
As in the proof of \cite[Theorem~4.1]{akd-pank}, we consider the case $\kappa=-1$ only. The other one being simpler.

 Now we will show that $-J_{\lambda}$ attains its unique (positive) maximum on ${\hat E}(w)\cap\mathcal{N}$ following the following steps
 \begin{itemize}
     \item  $-J_{\lambda}$ attains its unique maximum on ${\hat E}(w)$.
     \item ${\hat E}(w)\cap\mathcal{N}\neq\emptyset$ for any $w\in E\setminus F=E\setminus
(E^0\oplus E^+)$.
\item the uniqueness of global maximum of $-J_{\lambda}$ on ${\hat E}(w)\cap\mathcal{N}$.
 \end{itemize}
Since ${\hat E}(w)={\hat E}(w^-/\|w^-\|)$, without loss of generality we may assume that $w\in E^-$ and $\|w\|=1$. We will show that there
exists $R>0$ such that $-J_\lambda(u)\leq 0$ for all $u\in {\hat E}(w)$ with $\|u\|\geq R$. Suppose to the contrary that we can find a sequence $\{u_n\}$ such that $\|u_n\|\to\infty$ and $-J_\lambda(u_n)\geq 0$. We set $v_n=\|u_n\|^{-1}u_n$. Passing to a subsequence, we may assume that $v_n\to
v$ weakly in the space $E$. And it is easily seen that the first two terms of the following relation
\begin{equation}\label{e4.2}
	0\leq \frac{-J_\lambda(u_n)}{\|u_n\|_E^2}=
	-\frac{q_\lambda(v_n^-)}{2}-\frac{q_\lambda(v_n^+)}{2}-\frac{\Psi(\|u_n\|_Ev_n)}{\|u_n\|^2_E}\,.
\end{equation}
are bounded. It folllows immediately from $(iv)$ that if $v\neq 0$, then the third term
tends to $-\infty$, a contradiction. Hence, $v=0$. It implies that $v_n^-\to 0$ and $v_n^0\hookrightarrow 
0$ weakly. Since $v^-_n$ and $v^0_n$ belong to finite
dimensional subspaces $E^-=\mathbb{R}w$ and $E^0$, respectively, $v_n^-\rightharpoonup  0$ and $v_n^0\rightharpoonup  
0$ weakly imply  $v_n^-\to 0$ and $v_n^0\to 0$ strongly.  From the inequalities
(\ref{e4.1}), (\ref{e4.2}), and since $\Psi\geq 0$,  $\|v_n^+\|_E\leq \|v_n^-\|_E\to 0$.

On the other hand,

$$
1=\|v_n\|^2_E=\|v_n^-\|^2_E+\|v_n^0\|^2_E+\|v_n^+\|^2_E\to 0\,,
$$
a contradiction.

Since $-J_\lambda\leq 0$ on ${\hat E}(w)\setminus B_R(0)$ for sufficiently large $R$, the boundedness of $\Psi$ on bounded subsets of $E$ implies  that
$\sup_{u\in{\hat E}(w)} (-J_\lambda(u))<\infty$. Using $(iii)$, we get $-J_\lambda(tw)=\gamma
t^2+o(t^2)$ as $t\to 0$, where $\gamma=-q_\lambda(w)/2>0$. Therefore, the supremum is positive. And we have $0<\sup_{u\in{\hat E}(w)} (-J_\lambda(u))<\infty$. We shall show that $-J_\lambda$ achieves its (positive) maximum value on ${\hat E}(w)$. To prove upper weakly semicontinuity of 
$-J_\lambda$ on $E(w)$ is sufficient for us. By $(i)$, $\Psi$ is weakly continuous, so it is
enough to prove that $q_\lambda$ is weakly \emph{low} semicontinuous on $E(w)$.
For this, let us take a weakly convergent sequence $u_n=t_nw+u_n^0+u^+_n\in E(w)$, which converges weakly in $E$. This implies  $t_n\to t$,
$u^+_n\to u^+$ weakly, and $q_\lambda(u_n)=q_\lambda(t_nw)+q_\lambda(u^+_n)$. Since
$q_\lambda|_{E(w)}$ is a positive definite continuous quadratic form, it is a convex continuous
function on $E(w)$, and so, it must be weakly low semicontinuous. Since $q_\lambda$ is weakly lower semicontinuous and it has a minimizing sequence $\{t_n w\}$, using the fact that $q_\lambda(t_nw)\to q_\lambda(tw)$, we can say that $q_\lambda$ has a minimum on  ${\hat E}(w)$. Because, $tw\in \mathcal{N}$ and so ${\hat E}(w)\cap\mathcal{N}\neq\emptyset$.

The uniqueness of global maximum of $-J_\lambda|_{{\hat E}(w)}$ can be obtained exactly as in
\cite[Proposition~39]{sz-weth-10} (see also \cite{sz-weth-09}).

\end{proof}

\section{Proof of Theorem~\ref{theorem2}}

   In this section we prove Theorem~\ref{theorem2}. The subspaces $E^-$, $E^0$, and $E^+$ is defined as in Section~\ref{s3}. Under our assumption, we can say that $E^0=\{0\}$.   Let us recall the spectral decomposition for a given $\lambda\not \in \sigma(L)$
    \begin{equation}\label{eq:energyplusminus}
        E= E^- \oplus E^+, \quad E^-=\bigoplus\limits_{1\le n\le k-1} \ker( L-  \lambda_n I),\quad E^+=\bigoplus\limits_{n\ge k} \ker( L- \lambda_n I). 
    \end{equation}
    The following lemma allows us a discrete version of well-known inequalities which are adapted version of \eqref{e4.1}. It may be useful to denote $\lambda_0=-\infty$. 
   
   \begin{lemma}\label{lemma6.1}
	Let $k\in \mathbb N$. If $ \lambda\in ( \lambda_{k-1}, \lambda_k)$, then  
	
	\begin{equation}\label{eq:lowerboundq}
	 q_{ \lambda} (u)\geq \dfrac{ \lambda_k- \lambda}{ \lambda_k}\|u\|_E^2\,,\quad u\in E^+,
	\end{equation}
	and for $k>1$ we have
		\begin{equation}\label{eq:upperboundq}
		 q_{ \lambda} (u)\leq \dfrac{{ \lambda}_{k-1}-{ \lambda}}{{ \lambda}_{k-1}}\|u\|_E^2\,,\quad u\in E^-\;\;.
		\end{equation}
\end{lemma}

\begin{proof}
	The proof is an easy adaptation of \cite[Lemma~6.1]{akd-pank}. Let $\{e_{n}\}$ be an orthonormal basis of eigenfunctions corresponding to the eigenvalues ${ \lambda}_n$. 
   
   Suppose $u\in E^+$, then there exists $\{a_{n}\} \subset \mathbb R$ such that
 \begin{equation}
    u =\sum_{n\ge k} a_{n} e_{n} \in E^+.
 \end{equation}
 Then by Parseval's theorem
 \begin{align*}
 \|u\|_{l^2_m}^2&=\sum_{n\ge k} a_{n}^2\\
  \|u\|_{E}^2&=  q(u,u)=\sum_{n\ge k} { \lambda}_na_{n}^2.
  \end{align*}
  Using the fact that $ q_{{ \lambda}}(u)= q(u)-{ \lambda}\|u\|_{\ell^2_m}^2,$ we get
 \begin{align*}
      q_{ \lambda}(u) &= \sum_{n\ge k}({ \lambda}_n- { \lambda})a_{n}^2 \nonumber\\\nonumber &=\sum_{n\ge k} \left ( 1 - \frac{{ \lambda}}{{ \lambda}_n}\right ) { \lambda}_n a_{n}^2 \nonumber \\ \nonumber&\ge \left ( 1- \frac{{ \lambda}}{{ \lambda}_k}\right ) \sum_{n\ge k} { \lambda}_n a_{n}^2 \\ \nonumber&= \frac{{ \lambda}_k-{ \lambda}}{{ \lambda}_k}  q(u).  
 \end{align*}
 Similarly for any $u \in E^-$ there exists $\{b_{n}\} \subset \mathbb R$ for $1\le n \le k-1$ and $1\leq j\leq m_n$ such that
 \begin{equation*}
    u = \sum_{1\le n\le k-1} b_{n}e_{n}
 \end{equation*}
 and we obtain similarly
\begin{equation*}
  q_{ \lambda}(u) \ge \frac{ \lambda_{k-1} -  \lambda}{ \lambda_{k-1}}  q(u).
\end{equation*}
Recall $\|u\| =  q(u)$ and we obtain \eqref{eq:lowerboundq} and \eqref{eq:upperboundq}.
\end{proof}


In the sequel, let us denote the orthogonal projectors in $E$ onto the subspaces $E^+$ and $E^-$ (defined in \eqref{eq:energyplusminus}) by $P^+$ and $P^-$, respectively. Then $P^-=I-P^+$, where $I$ stands for the identity operator. Note also that $P^+$ and $P^-$ extend to orthogonal projectors in $\ell^2_m$. 

\begin{lemma}\label{lem:6.2}
	Suppose \eqref{eq:replacesexpgrowth} holds.  Let ${ \lambda}\in ({ \lambda}_{k-1},{ \lambda}_k)$ for $k>1$, then suppose \eqref{Sobolev_Assumption} holds, then 
    the projectors $P^+$ and $P^-$ are bounded operators with respect to  $\|\cdot\|_{\ell^p_m}$.
\end{lemma}
\begin{proof}
    By Lemma~\ref{lem:sobolev1} we have $E\hookrightarrow \ell^p$. 
	Since $P^+=I-P^-$, without loss of generality, we consider the projector $P^-$ only. 
	
	Since $E^-$ is finite dimensional, if $N=\dim E^-$, then one can have an $\ell^2_m$- orthonormal basis  $\{e_1,e_2,\ldots,e_N\}$ in $E^-$.
	
	Thus, for each $u\in E$, $P^-u$ has the following expression
	$$P^-u=\sum\limits_{j=1}^{N}\xi_je_j.$$
	Since all norms in $E^-$ are equivalent, it is enough to show that 
	$$|\xi_j|\leq C \|u\|_{\ell^p_m},\; j=1,\ldots,N,$$
	for any constant $C>0$ independent of $j$ and $u$. Clearly, 
	$$\xi_j=(u,e_j)_{\ell^2_m},\;j=1,\ldots,N.$$
	Since $u\in E\subset \ell^p_m$ and $e_j\in \ell^{p{'}}_m$ for $j=1,\ldots,N$. 
	With H\"{o}lder's inequality we have
    \begin{equation}
        (u, e_j)_{\ell^2_m} \le \|u\|_{ \ell^p_m}\|e_j\|_{ \ell^p{'}_m}, 
    \end{equation}
    that implies the desired one.
	\end{proof}
    Recall in the following
    \begin{displaymath}
        \delta( \lambda) := \operatorname{dist}( \lambda, \sigma( L)). 
    \end{displaymath}
	Assuming, in the lemma below, we get estimates for the $\ell^p_m$- norms of a critical point in terms of its critical values.
	\begin{lemma}\label{lem:6.3}
		In addition to assumptions $(f_1)-(f_4)$, assume that the nonlinearity satisfies \eqref{AR}  for some $q>2$
		and \eqref{6.1} for some $p\ge 2, a_0 >0$. 
 Let $ \lambda\in ( \lambda_{k-1}, \lambda_k)$ with $k>1$. 
        \begin{itemize}
        \item Then there exists a constant $C>0$ such that for any critical point $u\in E$ of $\kappa J_{\lambda}$,
		\begin{equation}\label{6.2}
		\|u\|_{\ell^p_m}^p\leq C \kappa J_{\lambda}(u).
		\end{equation}
		\item Assume, in addition, that the requirements of Lemma~\ref{lem:6.2}  are satisfied, and \eqref{6.3} holds
         for some $a_1>0$.
		Then
		\begin{equation}\label{6.4}
		\delta(\lambda)\|u\|_E^2\leq C\kappa J_{\lambda}(u).
		\end{equation}
        \end{itemize}
	\end{lemma}
	\begin{proof}
		Consider the case $\kappa=1$, as the other one can be obtained analogously. From \eqref{AR}, we get
			\begin{equation}\label{6.5}
			J_{\lambda}(u)\geq (2^{-1}-q^{-1})\sum\limits_{x\in \mathcal{V}}m(x)f(x,u(x))u(x).
			\end{equation}
		Using both \eqref{AR} and \eqref{6.1} we have
            \begin{align*}
                J_{\lambda(u)} &\ge (2^{-1} - q^{-1}) \sum_{x\in \mathcal V} m(x) f(x,u) u\\
                &\ge (2^{-1} - q^{-1} ) \sum_{x\in\mathcal V} m(x) q F(x,u)\\
                &\ge q (2^{-1} - q^{-1}) a_0 \sum_{x\in \mathcal V} m(x) |u|^p
            \end{align*}
			and \eqref{6.2} follows immediately. 
            
            Now, let us prove the remaining part of the lemma. For now assume that $k>1$. For $u\in E$ denote in the following
            \begin{equation}
                u^+ := P^+ u, \qquad u^- := P^- u
            \end{equation}
            the projections of $u$ onto $E^+$ and $E^-$, respectively. Then
			$$0=\nabla J_{\lambda}(u)u^{+}=q_{\lambda}(u^{+})-\sum\limits_{x\in \mathcal{V}}m(x)f(x,u)u^{+},$$
			and
		$$0=\nabla J_{\lambda}(u)u^{-}=q_{\lambda}(u^{-})-\sum\limits_{x\in \mathcal{V}}m(x)f(x,u)u^{-}.$$ From Lemma~\ref{lemma6.1}, we obtain that
        \begin{align}\label{eq:plusest}
        \sum\limits_{x\in \mathcal{V}}m(x)f(x,u)u^{+}&\geq \dfrac{\lambda_k-\lambda}{\lambda_k}\|u^+\|_
        E^2 \ge \dfrac{\delta(\lambda)}{\lambda_k}\|u^+\|_
        E^2, \\ \label{eq:minusest} 
        -\sum\limits_{x\in \mathcal{V}}m(x)f(x,u)u^{-}&\geq \dfrac{\lambda-\lambda
        _k}{\lambda_{k-1}}\|u^-\|_
        E^2 \ge \dfrac{\delta(\lambda)}{\lambda_k}\|u^-\|_
        E^2. 
        \end{align}

		by the spectral decomposition theorem we have $\| u \|_E^2= \| u^+\|_E^2 + \|u^-\|_E^2$. If we add \eqref{eq:plusest} and \eqref{eq:minusest} together, we get together with \eqref{6.3}
		\begin{align*}
			\dfrac{\delta(\lambda)}{\lambda_k}\|u\|_E^2 &\leq   \sum\limits_{x\in \mathcal{V}}m(x)|f(x,u)||u^{+}|+\sum\limits_{x\in \mathcal{V}}m(x)|f(x,u)||u^{-}|\\
		&\leq	a_1 \sum\limits_{x\in \mathcal{V}}m(x)|u|^{p-1}|u^{+}|+a_1 \sum\limits_{x\in \mathcal{V}}m(x)|u|^{p-1}|u^{-}|.
		\end{align*}
		Using H\"{o}lder's inequality with mutually conjugate exponents $q$ and $p$, we get
		
		$$	\dfrac{\delta(\lambda)}{\lambda_k}\|u\|_E^2 \leq a_1 \|u\|^{p/q} _{\ell^p_m}(\|u^{+}\|_{\ell^p_m}+\|u^{-}\|_{\ell^p_m}).$$
		From Lemma~\ref{lem:6.2}, we conclude that $P^+$ and $P^-$ are bounded in $\ell^p_m$, that is,
		$$\|u^{+}\|_{\ell^p_m}+\|u^{-}\|_{\ell^p_m}\leq C \|u\|_{\ell^p_m}.$$
		Then with \eqref{6.2} we obtain \eqref{6.4}.
		
		The case $k=1$ is analogue, since in this case $E^{-}$ will be trivial, and we have $u=u^+$. In particular, Lemma~\ref{lem:6.2} is not needed.
 		\end{proof}

 		\begin{proof}[Proof of Theorem~\ref{theorem2}]
 			Let $\kappa=1$ and assume that $\lambda< \lambda_k$ such that $\delta(\lambda)=\lambda_k-\lambda$. Choose an eigenvector $e\in D(L)$ corresponding to the eigenvalue $\lambda_k$ with $\|e\|_{\ell^p_m}=1$. 
            
            Let $u=te+w$, where $t>0$, $w\in E$. Then since $q_{\lambda}(u)=q_{\lambda}(u^+)+q_{\lambda}(u^-)$ for each $u\in E$, we have $q_{\lambda}(u)=t^2q_{\lambda}(e)+q_{\lambda}(w)$. Then by Remark~\ref{r3.3}, we get 
            \begin{align}
            c_{\lambda}:=c= \inf_{w\in {E^-} }\max_{u\in{\hat {E}}(e)} J_{\lambda}(u)&\leq \max_{u\in{\hat{E}}(e)} J_{\lambda}(u)\nonumber\\  \nonumber&=\max_{t>0,\;w\in {E^-}} \left(\dfrac{t^2q_{\lambda}(e)}{2}-\dfrac{q_{\lambda}(w)}{2}-\sum\limits_{x\in\mathcal{V}}m(x)F(x,te+w)\right).
            \end{align}
Since $q_{\lambda}(e)=\lambda_k-\lambda$ and $q_{\lambda}(w)\leq 0,$ using \ref{6.1}, we get
        $$c\leq \max_{t>0,\;w\in {E^-}} \left(\dfrac{t^2( \lambda_k -  \lambda)}{2}-a_0\|te+w\|_{\ell_p^m}^p\right).$$
        Since $P^+$ is a bounded operator on a finite dimensional space $E(e)$ with respect to $\ell^p_m$ norm from Lemma~\ref{lem:6.2}, then since $P^+(w+ te)= te$ there exists $b_1>0$ such that
        $$\|te+w\|_{\ell^p_m}^p\geq b_1\|e\|_{\ell^p_m}^p.$$
        Thus, we get
         $$ c\leq \max_{t>0} \left(\dfrac{t^2( \lambda_k -  \lambda)}{2}-a_0 b_2t^p\right).$$
         With an easy calculation we obtain
\begin{equation}\label{6.7}
    c\leq C ( \lambda_k -  \lambda)^{\frac{p}{p-2}},
\end{equation}         
            for some $C>0$ depending only on $p$ and $b_2$.

From Theorem~\ref{existenceofcp}, there exists $u_{ \lambda}\in E$ such that $c_{ \lambda}=J_{ \lambda}(u_{ \lambda})$. Combining \eqref{6.4} with \eqref{6.7} and obtain
            \begin{equation*}
                \frac{( \lambda_k -  \lambda)\|u_\lambda\|_E^2}{C} \le J_\lambda( u_\lambda) = c \le C( \lambda_k -  \lambda)^{\frac{p}{p-2}}.
            \end{equation*}
            Hence,
            \begin{equation*}
                \| u_\lambda\|^2_E \le C^2 (\lambda_k - \lambda)^{\frac{2}{p-2}},
            \end{equation*}
            which implies \eqref{eq:bifurcationfirst}.  

            Let us now consider the case $\kappa =-1$. Assume that $\lambda \in (\lambda_{k-1}, \lambda_k)$ such that $\delta(\lambda) = \lambda- \lambda_{k-1}$. In this case $c= -J_\lambda(u_\lambda)$ and suppose $e$ is an eigenfunction to $\lambda_{k-1}$ with $\|e\|_{\ell^p_m} =1$. Then we obtain
            \begin{equation*}
            c= \inf_{w\in E^+} \max_{t>0} J_{ \lambda} (w + te) \leq \max_{t>0,\;w\in {E^+}} \left(\dfrac{t^2(\lambda-\lambda_{k-1} )}{2}-a_0\|w+te\|_{\ell_p^m}^p\right).
            \end{equation*}
            $P_-$ is a bounded operator on $\ell^p_m$ by Lemma~\ref{lem:6.2} and there exists $b_3>0$ such that
            \begin{equation*}
                \| e+t w\|_{\ell^p_m}^p \le b_3 \| e\|_{\ell^p_m}^p. 
            \end{equation*}
            Then similarly we obtain
            \begin{equation*}
                c \le \max_{t>0} \left ( \frac{t^2 ( \lambda-  \lambda_{k-1})}{2} - a_0 b_3 t^p \right ),
            \end{equation*}
            and one shows easily similar as before \eqref{eq:bifurcationsecond}.
            
 			\end{proof}

\appendix
\section{Critical Point Theory}\label{s7}
	
\normalfont

 Now, we give some of the results from \cite{sz-weth-10}. For this, first recall some related concepts: \\
 
  Let $\mathcal{E}$ be an abstract Hilbert space and $Q(u,v)$ be a bounded, symmetric
	bilinear form on $\mathcal{E}$. Let $Q(u)=Q(u,u)$ be the associated quadratic form. With respect to the associated quadratic form $Q(u)$, we decompose the Hilbert space $\mathcal{E}$ as follows 
	$$
	\mathcal{E}=\mathcal{E}^-\oplus \mathcal{E}^0\oplus E^+\,,
	$$
	here the form $Q$ is negative definite on
	$\mathcal{E}^{-}$, positive definite on $\mathcal{E}^+$ and $Q=0$ on $\mathcal{E}^0$. Thus, for any $ u, v\in \mathcal{E}$ we have
	$$
	u=u^-+u^0+u^+, \;\;\;\; Q(u,v)=Q(u^+,v^+)+Q(u^-,v^-),
	$$
	where $u^-, v^-\in \mathcal{E}^-$, $u^0\in \mathcal{E}^0$ and $u^+, v^+\in \mathcal{E}^+$. 
	
	On the space $\mathcal{E}$, we have a functional $J$ of the form
	$$
	J(u)=\frac{1}{2}Q(u)-\Phi(u)\,,
	$$
	where $\Phi$ is a $C^1$ functional on $\mathcal{E}$ such that $\Phi(0)=0$. The derivative of $J$ at $u$ is
	given by
	$$
	\nabla J(u)v=Q(u,v)-\nabla \Phi(u)v
	$$
	for all $v\in \mathcal{E}$.\\
	
	\noindent Set $F=\mathcal{E}^-\oplus \mathcal{E}^0$. For the existence of nontrivial critical points, we also need to define
	$$
	\mathcal{E}(u)=\mathbb{R}u\oplus \mathcal{F}=\mathbb{R}u^+\oplus \mathcal{F}\,,
	$$
	and
	$$
	{\hat {\mathcal{E}}}(u)=\{tu+v: t\geq 0, v\in \mathcal{F}\}=\{tu^++v: t\geq 0, v\in \mathcal{F}\}\,.
	$$
	In this section  we suppose the following
	assumptions hold:
	\begin{itemize}\em
		\item[(i)] The functional $\Phi$ is weakly lower semicontinuous and
		\begin{equation}\label{e3.1}
		\frac{1}{2}\nabla\Phi(u)u>\Phi(u)>0\;,\;\;\forall u\neq 0.
		\end{equation}
		
		\item[(ii)] For each $w\in \mathcal{E}\setminus \mathcal{F}$ the functional $J|_{{\hat {\mathcal{E}}}(w)}$ has  a unique
		nonzero critical point $m(w)\in {\hat {\mathcal{E}}}(w)$. At that point, $J|_{{\hat {\mathcal{E}}}(w)}$ achieves its global maximum.

	\end{itemize}
	
	\vspace{.2cm}
	
	The {\em generalized Nehari manifold} of the functional $J$ is defined by
	$$
	\mathcal{N}=\mathcal{N}(J)=\{u\in \mathcal{E}\setminus \mathcal{F} : \nabla J(u)u=0 \mbox{ and } \nabla J(u)v=0 \mbox{ for all }
	v\in \mathcal{F}\}\,.
	$$
	If $u\neq 0$ is a critical point of $J$, then 
	$$
	J(u)=J(u)-\frac{1}{2}\nabla J(u)u= \frac{1}{2}\nabla \Phi(u)u-\Phi(u)>0\,.
	$$
	But, the critical value $J(u)\leq 0$ for each $u\in \mathcal{F}$. This implies that $\mathcal{N}$ contains all nonzero
	critical points of $J$. Furthermore, if $w\in \mathcal{E}\setminus \mathcal{F}$, then $\mathcal{N}\cap {\hat {\mathcal{E}}}(w)$ consists
	of exactly one point $m(w)$, that is, $\mathbf{\mathcal{N}}=\{m(w) : w\in \mathcal{E}\setminus \mathcal{F}\}$.

	
	We will use the next result from  \cite{sz-weth-10} (see Theorem 35) to prove the existence of solutions. 
	
	\begin{theorem}\label{t3.2}
		We suppose $(i)$ and $(ii)$, and the following assumptions 
		\begin{itemize}
			\item[$(iii)$]  $\nabla\Phi(u)=o(\|u\|_{\mathcal{E}})$ as $u\to 0$;
			\item[$(iv)$] $\Phi(tu)/t^2\to \infty$ as $t\to\infty$ uniformly for $u$ on weakly compact subsets of
			$\mathcal{E}\setminus\{0\}$;
			\item[$(v)$] $\nabla\Phi$ is completely continuous.
		\end{itemize}
		are satisfied. Then
		$$
		c=\inf_{u\in\mathcal{N}} J(u) >0
		$$
		is a nontrivial critical value of $J$. Additionally, if $\Phi$ is even, then the equation $\nabla J(u)=0$ has	at least $\dim \mathcal{E}^+$ pairs of nontrivial solutions.
	\end{theorem}

	\begin{remark}\label{r3.3}
		The infimum of $J$ over $\mathcal{N}$ has the following minimax characterization
		$$
		c = \inf_{w\in \mathcal{E}\setminus \mathcal{F}} \max_{u\in{\hat{\mathcal{E}}}(w)} J(u)=\inf_{w\in \mathcal{E}^+, \;\|w\|=1}
		\max_{u\in{\hat {\mathcal{E}}}(w)} J(u)\,.
		$$
			The critical value $c$ and the corresponding critical points are called \emph{ground level} and 
			\emph{ground critical points} of the functional, respectively.
	\end{remark}

\end{document}